\documentclass{article}

\usepackage{arxiv}
\usepackage[utf8]{inputenc}
\usepackage[T1]{fontenc}    
\usepackage{hyperref}       
\usepackage{url}            
\usepackage{booktabs}       
\usepackage{amsfonts}       
\usepackage{nicefrac}   
\usepackage{microtype}
\usepackage{graphicx}
\usepackage{doi}
\usepackage{amsmath}
\usepackage{amsthm}
\usepackage{amssymb}
\usepackage{graphicx}
\usepackage{tikz-cd}
\usepackage{tikz}
\usetikzlibrary{matrix, arrows}
\tikzset{node distance=2cm, auto}
\newtheorem{theorem}{Theorem}[section]
\newtheorem{proposition}[theorem]{Proposition}

\newtheorem{definition}[theorem]{Definition}
\newtheorem{remark}[theorem]{Remark}

\title{A Logical Analysis of Universal Properties}

\author{Talal H. Alrawajfeh\\talalalrawajfeh@gmail.com}

\renewcommand{\shorttitle}{A Logical Analysis of Universal Properties}

\hypersetup{
    pdftitle={A Logical Analysis of Universal Properties},
    pdfsubject={q-bio.NC, q-bio.QM},
    pdfauthor={Talal H. Alrawajfeh},
    pdfkeywords={Logic, Category Theory, Universal Property},
}

\begin{document}
\maketitle

\begin{abstract}
    A Universal Mapping Property is generally described as a characterization of an object up to a unique isomorphism by considering its relation to every other object; however, the term "by considering its relation to every other object" is not clearly or explicitly defined. In this paper, we will introduce such definition which will also generalize the idea of a universal property from a logical perspective.
\end{abstract}


\section{Introduction}
In Category Theory, a structure of objects together with arrows is uniquely determined up to unique isomorphism through the relations these objects and arrows have with other objects and arrows via a Universal Mapping Property (UMP) (see \cite{awodey2010} and \cite{maclane1971}). An example of a UMP is the UMP of a product in a category:

In a category $\mathcal{C}$, the product of $A$ and $B$ consists of an object $A \times B$ and two arrows $\pi_A : A \times B \rightarrow A$ and $\pi_B : A \times B \rightarrow B$, usually called projections, such that given any diagram of the form:
\begin{center}
	\begin{tikzcd}
		A & X \arrow[r, "p_B"] \arrow[l, "p_A"'] & B
	\end{tikzcd}
\end{center}
there is a unique arrow $u : X \dashrightarrow A \times B$ that makes the following diagram commute.
\begin{center}
	\begin{tikzcd}[transform shape, nodes={scale=1.2}]
		  & X \arrow[ld, "p_A"'] \arrow[rd, "p_B"] \arrow[d, "u", dashed] &   \\
		A & A \times B \arrow[r, "\pi_B"] \arrow[l, "\pi_A"']                      & B
	\end{tikzcd}
\end{center}

A relation here between the structures 
\begin{tikzcd}
	A & A \times B \arrow[r, "\pi_B"] \arrow[l, "\pi_A"'] & B
\end{tikzcd}
and 
\begin{tikzcd}
	A & X \arrow[r, "p_B"] \arrow[l, "p_A"'] & B
\end{tikzcd}
is that the arrows $\pi_A$ and $p_A$ share the same codomains and the arrows $\pi_B$ and $p_B$ also share the same codomains. Another relation is the existence of a unique arrow $u$ from $X$ to $A \times B$ that satisfies $\pi_B \circ u = p_B$ and $\pi_A \circ u = p_A$. The existence of a unique arrow $u: X \dashrightarrow A \times B$ implies that $A \times B$ is unique up to unique isomorphism. To demonstrate this, assume that there is another object $P$ that is a product of $A$ and $B$ together with the arrows ${\pi'}_A: P \rightarrow A$ and  ${\pi'}_B: P\rightarrow B$, then there exists unique arrows $u_1: A \times B \rightarrow P$ and $u_2: P \rightarrow A \times B$ making the following diagram commute:

\begin{center}
	\begin{tikzcd}[transform shape, nodes={scale=1.2}]
		  & A \times B \arrow[ld, "\pi_A"'] \arrow[rd, "\pi_B"] \arrow[d, "u_1", dashed] &   \\
		A & P \arrow[r, "{\pi'}_B"] \arrow[l, "{\pi'}_A"']  \arrow[d, "u_2", dashed]   & B \\
		  & A \times B \arrow[lu, "\pi_A"] \arrow[ru, "\pi_B"']  &
	\end{tikzcd}
\end{center}

Also, since there is a unique arrow from $A \times B$ to itself by letting $X = A \times B$ and $p_A = \pi_A$ and $p_B = \pi_B$, and from the fact that there is an identity arrow $1_{A \times B}$ on $A \times B$, then $u_2 \circ u_1: A \times B \rightarrow A \times B$ is $1_{A \times B}$. Similarly, $u_1 \circ u_2: P \rightarrow P$ must be $1_P$. Hence, $u_1$ is the unique isomorphism from $A \times B$ to $P$.

In this paper, we will attempt to analyze and generalize universal properties from a logical point of view.

\section{Motivation}

We will start by the simplest possible way of defining what is meant by the characterization of an object uniquely through its relation to every other object. We will also start with the simplest meaning of "uniqueness" which is in terms of equality. To avoid paradoxes from statements involving terms like "all sets" or "all groups", the concept of a Class is utilized either informally under ZFC set theory or formally under BG set theory (see \cite{jech2002} and \cite{borceux1994}).

\begin{definition}\label{universal1}
We say that an object $u$ is $R$-universal if for a binary relation $R$ on a class $\mathcal{C}$, $u$ is the unique object that satisfies
\[
    \forall x : R(x, u)
\]

Equivalently,
\[
    \forall x : R(x, u) \wedge \forall v : \left[\forall x : R(x,v)\right] \implies v = u
\]
\end{definition}

As an example, consider the natural numbers $\mathbb{N}$ with the order relation $<$. Let $R(a,b) := a > b$, then $u = 1$ is $R$-universal since $1$ is the least element in $\mathbb{N}$ and if we assume that there is another $v \neq 1$ that is $R$-universal we arrive at a contradiction $v < 1$ so we must conclude that $v = 1$. 

It must be noted here that when we say an object $u$ is $R$-universal we don't necessarily mean that $u$ is an element of a set or a class but may be a set or a map or a structure (tuple) of sets and maps or in the case of Category Theory it could be a structure (tuple) of objects and arrows.

Although definition ~\ref{universal1} captures the idea of an object $u$ characterized by its relation to every other object, it is too restrictive because of the requirement that any object satisfying the same characterization must be identical to $u$.

To see how we can relax the condition of equality, consider, for example, in the context of propositional logic, the relation of logical equivalence ($\equiv$) for which we can use instead of equality:

\begin{equation}\label{universal-equivalence}
    \forall x : R(x, u) \wedge \forall v : \left[\forall x : R(x,v)\right] \implies v \equiv u
\end{equation}

Furthermore, we can replace logical equivalence with logical implication ($\implies$ or its dual $\impliedby$):

\begin{equation}\label{universal-implication}
    \forall x : R(x, u) \wedge \forall v : \left[\forall x : R(x,v)\right] \implies (v \implies u)
\end{equation}

To demonstrate that ~\ref{universal-implication} implies that any two objects $u_1$ and $u_2$ satisfying the same property are equivalent, observe that we can substitute $u_1$ for $v$ in $\forall x : R(x,v)$ when $u$ is $u_2$ which gives $u_1 \implies u_2$. Similarly, substituting $u_2$ for $v$ when $u$ is $u_1$, we obtain $u_2 \implies u_1$; hence, $u_1 \equiv u_2$. Therefore, ~\ref{universal-implication} and ~\ref{universal-equivalence} are equivalent.

In the context of Groups or Category Theory, we can replace equality with the existence of a unique isomorphism ($\cong$):

\begin{equation}\label{universal-iso}
    \forall x : R(x, u) \wedge \forall v : \left[\forall x : R(x,v)\right] \implies v \cong u
\end{equation}

Moreover, we can replace $\cong$ with a unique homomorphism/arrow ($\dashrightarrow$ or its dual $\dashleftarrow$):

\begin{equation}\label{universal-unique-arrow}
    \forall x : R(x, u) \wedge \forall v : \left[\forall x : R(x,v)\right] \implies \exists ! m : v \dashrightarrow u
\end{equation}

To show that ~\ref{universal-unique-arrow} implies that any two objects $u_1$ and $u_2$ satisfying the same property are isomorphic $u_1 \cong u_2$ and that the isomorphism is unique, we can carry out a similar argument as we did previously with ~\ref{universal-implication} to obtain the existence of unique homomorphisms/arrows $f: u_1 \dashrightarrow u_2$ and $g: u_2 \dashrightarrow u_1$. Also, substituting $u_1$ for $v$ in $\forall x : R(x,v)$ when $u$ is $u_1$ implies that there is a unique homomorphism/arrow $u_1 \dashrightarrow u_1$ which must be the identity $1_{u_1}$; hence, $g\circ f = 1_{u_1}$. Similarly, the homomorphism/arrow $u_2 \dashrightarrow u_2$ must be $1_{u_2}$; hence, $f\circ g = 1_{u_2}$, and thus $f$ is the unique isomorphism from $u_1$ to $u_2$. Therfore, ~\ref{universal-unique-arrow} and ~\ref{universal-iso} are equivalent.

In the next section, we introduce a generalization of definition ~\ref{universal1} that best captures the ideas we discussed in ~\ref{universal-equivalence}, ~\ref{universal-implication}, ~\ref{universal-iso}, and ~\ref{universal-unique-arrow}.

\section{A Generalization}

A preorder is a reflexive transitive binary relation usually denoted $\preccurlyeq$ (see \cite{bergman2015}). An equivalence relation induced by $\preccurlyeq$ is the relation:
\[
a \approx b := (a \preccurlyeq b) \wedge (b \preccurlyeq a)
\]

One can define a preorder $\preccurlyeq$ on a set $X$ via an equivalence relation $\approx$ by defining a partial order $\leq$ (reflexive, transitive, and antisymmetric binary relation) on the set of equivalence classes $X/\approx$ (see \cite{bergman2015}) as follows:
\[
    a \preccurlyeq b \iff [a] \leq [b]
\]

Now, we define an $R$-universal object with respect to a preorder $\preccurlyeq$.

\begin{definition}\label{universal2}
We say that an object $u$ is $R$-universal w.r.t. a preorder $\preccurlyeq$ on a class $\mathcal{C}$ if for a binary relation $R$ on $\mathcal{C}$,
\[
    \forall x : R(x, u) \wedge \forall v : \left[\forall x : R(x,v)\right] \implies v \approx u
\]
where the relation $\approx$ is the equivalence relation induced by the preorder $\preccurlyeq$.
\end{definition}

We could obtain many different special cases of definition ~\ref{universal2} according to the required application. One special case that stands out is in the following remark.

\begin{remark}\label{universal2-ump}
A special case of definition ~\ref{universal2} could be obtained by setting, for any other relation $Q$:
\[
    R(a, b) := Q(a, b) \implies a \preccurlyeq b
\]
or dually,
\[
    R(a, b) := Q(a, b) \implies a \succcurlyeq b
\]

So, in this case, an object $u$ is $R$-universal w.r.t. a preorder $\preccurlyeq$ on a class $\mathcal{C}$ if:
\begin{equation}
\left[ \forall x : Q(x, u) \implies x \preccurlyeq u \right] \wedge \forall v : \left[\forall x : Q(x,v) \implies x \preccurlyeq v \right] \implies v \approx u
\end{equation}
or dually, 
\begin{equation}
    \left[ \forall x : Q(x, u) \implies x \succcurlyeq u \right] \wedge \forall v : \left[\forall x : Q(x,v) \implies x \succcurlyeq v \right] \implies v \approx u
\end{equation}
\end{remark}

The special form in remark ~\ref{universal2-ump} of definition ~\ref{universal2} is closer to the form of a UMP. For example, if we consider the UMP of a product in a category, $u$ is the product $A \times B$ of the objects $A$ and $B$ together with the arrows $\pi_A : A \times B \rightarrow A$ and $\pi_B : A \times B \rightarrow B$. The relation $Q$ is the relation that any structure of the form
\begin{tikzcd}
	A & X \arrow[r, "p_B"] \arrow[l, "p_A"'] & B
\end{tikzcd}
has with the structure
\begin{tikzcd}
	A & A \times B \arrow[r, "\pi_B"] \arrow[l, "\pi_A"'] & B
\end{tikzcd}
, that is, $p_A$ and $p_B$ share the same codomains as $\pi_A$ and $\pi_B$ respectively. The preorder $\preccurlyeq$ between the structures is the existence of a unique arrow $u:X \dashrightarrow A \times B$ that satisfies $\pi_A \circ u = p_A$ and $\pi_B \circ u = p_B$.

\section{Properties instead of Relations}
One can argue that a relation $R$ might implicitly encode a property or a quality, in the general sense; however, to make it explicit, we first have to discuss how relations might arise from properties.

A simple possible way of defining a relation $R$ from a property $P$ is by making two things related to each other if they have the same property:
\[
    R(a, b) := P(a) \wedge P(b)
\]

where $P(\_)$ is the predicate corresponding to the property $P$. For example, in the previous discussion of the UMP of the product $A \times B$ of two objects $A$ and $B$, the structure of the object $A \times B$ together with the arrows $\pi_A$ and $\pi_B$ has the property that the arrow $\pi_A$ has the object $A \times B$ as its domain and has the object $A$ as its codomain and the arrow $\pi_B$ has the object $A \times B$ as its domain and has the object $B$ as its codomain. If we denote the structure
\begin{tikzcd}
	A & A \times B \arrow[r, "\pi_B"] \arrow[l, "\pi_A"'] & B
\end{tikzcd}
briefly as a tuple the $(A \times B, \pi_A, \pi_B)$, then we can define a predicate $P(Y, f, g)$ as:
\[
    \text{dom}(f) = Y \wedge \text{dom}(g) = X \wedge \text{cod}(f) = A \wedge \text{cod}(g) = B
\]

where $\text{dom}(f)$ denotes the domain of $f$ and $\text{cod}(f)$ denotes the codomain of $f$. Hence, the relation $Q$ in the discussion after remark ~\ref{universal2-ump} of the product $A \times B$ could be written as:
\[
    Q((X, p_A, p_B), (A \times B, \pi_A, \pi_B)) := P(X, p_A, p_B) \wedge P(A \times B, \pi_A, \pi_B)
\]

Another way to define a relation, is by making two things related to each other if one has the property and the other doesn't:
\[
    R(a, b) := P(a) \wedge \neg P(b)
\]

In general, a relation $R(a, b)$ defined through a property $P$ is a statement $\phi$ about $P(a)$ and $P(b)$:
\[
    R(a, b) := \phi(P(a), P(b))
\]
which brings us to the next definition.

\begin{definition}\label{universal3}
    We say that an object $u$ is $P$-universal w.r.t. a preorder $\preccurlyeq$ on a class $\mathcal{C}$ if for a predicate $P$ on $\mathcal{C}$ and a statement $\phi$:
    \[
        P(u) \wedge \forall x : \phi(P(x), P(u)) \wedge \forall v : \left[\forall x : \phi(P(x), P(v))\right] \implies v \approx u
    \]
    where the relation $\approx$ is the equivalence relation induced by the preorder $\preccurlyeq$.
\end{definition}

Usually, a UMP characterizes an object uniquely, in some sense, by a property it shares with other objects. Hence, the following special case of definition ~\ref{universal3}.

\begin{remark}\label{universal3-ump}
    A special case of definition ~\ref{universal3} could be obtained by setting:
    \[
        \phi(P(a), P(b)) = P(a) \wedge P(b)
    \]

    So, in this case, an object $u$ is $P$-universal w.r.t. a preorder $\preccurlyeq$ on a class $\mathcal{C}$ if:
    \begin{equation}
    P(u) \wedge \left[ \forall x : P(x) \wedge P(u) \implies x \preccurlyeq u \right] \wedge \forall v : \left[\forall x : P(x) \wedge P(v) \implies x \preccurlyeq v \right] \implies v \approx u
    \end{equation}
    or dually, 
    \begin{equation}
        P(u) \wedge \left[ \forall x : P(x) \wedge P(u) \implies x \succcurlyeq u \right] \wedge \forall v : \left[\forall x : P(x) \wedge P(v) \implies x \succcurlyeq v \right] \implies v \approx u
    \end{equation}
\end{remark}

Now we will see that we can reduce the special form in remark ~\ref{universal3-ump} of definition ~\ref{universal3} to a more compact form as stated in the following proposition.

\begin{proposition}\label{universal3-ump-prop}
    Assume that $\phi(P(a), P(b)) = P(a) \wedge P(b)$. An object $u$ is $P$-universal w.r.t. a preorder $\preccurlyeq$ on a class $\mathcal{C}$ iff
    \[
        P(u) \wedge \forall x : P(x) \implies x \preccurlyeq u
    \]
    
    or dually,
    \[
        P(u) \wedge \forall x : P(x) \implies x \succcurlyeq u
    \]
\end{proposition}
\begin{proof}
    First assume that $u$ is $P$-universal w.r.t. a preorder $\preccurlyeq$ on a class $\mathcal{C}$, then that implies:
    \[
        P(u) \wedge \forall x : P(x) \wedge P(u) \implies x \preccurlyeq u
    \]
    
    Since $P(u)$ is independent of $x$ and always true, we conclude that
    \[
        P(u) \wedge \forall x : P(x) \implies x \preccurlyeq u
    \]

    A similar argument could be carried out for the dual case. Conversely, assume that the necessary condition holds. Since $P(u)$ is true, then the statement $P(x) \implies x \preccurlyeq u$ is equivalent to $P(x) \wedge P(u) \implies x \preccurlyeq u$, and thus
    \begin{equation}\label{universal3-ump-prop-necessary-condition}
        \forall x: P(x) \wedge P(u) \implies x \preccurlyeq u
    \end{equation}
    
    Assume that there is another object $v$ such that 
    $\forall x : P(x) \wedge P(v) \implies x \preccurlyeq v$, if we substitute $u$ for $x$, then
    \[
        P(u) \wedge P(v) \implies u \preccurlyeq v
    \]
    Moreover, substituting $v$ for $x$ in ~\ref{universal3-ump-prop-necessary-condition}, we obtain that
    \[
        P(v) \wedge P(u) \implies v \preccurlyeq u
    \]
    Therefore, $P(u) \wedge P(v) \implies u \approx v$. Similarly for the dual case.
\end{proof}

The necessary condition of proposition ~\ref{universal3-ump-prop} could be considered as a generalization of an optimization problem where $P(x)$ means that $x$ is a feasible solution and $u$ is the optimal solution (minimal if $x \succcurlyeq u$ and maximal if $x \preccurlyeq u$).

\section{Discussion}
UMPs are abundant in Mathematics and offer an alternative approach to define a mathematical object by its relations to other objects. These relations are structural in the sense that they pertain to the way the arrows and objects are arranged in a diagram and how different compositions of arrows are equal to one another. However, UMPs could be understood from a different perspective and extended to cover many other mathematical objects that are not necessarily defined by arrows or maps. We provided such definition and several special cases of it.

\bibliographystyle{unsrtnat}

\end{document}